\RequirePackage{fix-cm}

\documentclass[smallextended]{svjour3}       
\smartqed  
\usepackage{graphicx}
\usepackage{amssymb}
\usepackage{amsmath}
\usepackage{stmaryrd}
\usepackage{graphicx}
\usepackage{placeins}
\usepackage{color}
\usepackage{upgreek}
\usepackage{pgfplots}
\usepackage{tikz}
\usetikzlibrary{shapes,arrows,calc,external}
\usepgfplotslibrary{external}

\def\ispace{j}
\def\imaxspace{J}

\def\itime{n}
\def\imaxtime{N}

\def\ifourier{k}

\def\lmax{\lambda_{\mathrm{max}}}
\def\Ahat{\widehat{A}}
\def\Atilde{\widetilde{A}}

\def\cflhat{\widehat{\nu}}

\def\cflmax{\nu_{\mathrm{max}}}

\def\cfli{\nu_1}
\def\cflii{\nu_2}

\def\eigAk{\mu}
\newcommand{\ww}{\underline w}

\newcommand{\dx}{\Delta x}
\newcommand{\dt}{\Delta t}
\def\dtdx{\frac \dt \dx}

\newcommand{\eps}{\varepsilon}
\newcommand{\HHn}{\widehat{\mathcal{H}}}
\newcommand{\HHs}{\widetilde{\mathcal{H}}}
\newcommand{\IMEX}{\mathcal{I}_\ispace^\itime}
\newcommand{\OO}{\Omega} 
\newcommand{\Aa}{\mathcal{A}_\ifourier}
\newcommand{\Nx}{\imaxspace}

\DeclareMathOperator{\Real}{Real}
\DeclareMathOperator{\diag}{diag}
\newcommand{\vecd}[2]{\left( \begin{matrix} #1 \\ \vdots \\  #2\end{matrix}\right)}
\DeclareMathOperator{\Id}{Id}
\DeclareMathOperator{\cfl}{CFL} 
\DeclareMathOperator{\CFL}{CFL} 

\DeclareMathOperator{\R}{\mathbb{R}}

\DeclareMathOperator{\Z}{\mathbb{Z}}

\newcommand{\vel}{\upupsilon}

\begin{document}

\title{Flux Splitting for stiff equations: A notion on stability}


\author{Jochen Sch\"utz \and Sebastian Noelle}

\institute{J. Sch\"utz and S. Noelle \at
           Institut f\"ur Geometrie und Praktische Mathematik, RWTH Aachen University \\ Templergraben 55, 52062 Aachen \\
              Tel.: +49 241 80 97677 \\
              \email{\{schuetz,noelle\}@igpm.rwth-aachen.de}           
}

\date{Received: date / Accepted: date}

\maketitle

\begin{abstract}
For low Mach number flows, there is a strong recent interest in the development and analysis of \emph{IMEX} (implicit/explicit) schemes, which rely on a splitting of the convective flux into stiff and nonstiff parts. A key ingredient of the analysis is the so-called \emph{Asymptotic Preserving} (AP) property, which guarantees uniform consistency and stability as the Mach number goes to zero. While many authors have focussed on asymptotic consistency, we study asymptotic stability in this paper: does an IMEX scheme allow for a CFL number which is independent of the Mach number? We derive a stability criterion for a general linear hyperbolic system. 
In the decisive eigenvalue analysis, the advective term, the upwind diffusion and a quadratic term stemming from the truncation in time all interact in a subtle way. 
As an application, we show that a new class of splittings based on characteristic decomposition, for which the commutator vanishes, avoids the deterioration of the time step which has sometimes been observed in the 
literature.
\end{abstract}



\section{Introduction, Underlying Equations And Flux Splitting}

In recent years there has been a renewed interest in the computation of singularly perturbed differential equations. These equations arise, e.g., in the simulation of low-speed fluid flows. Here one is interested in computing waves with vastly different speeds. The goal is to resolve slow waves accurately and efficiently with a large time step, while approximating the fast waves in a stable way, using the same time step.

There is vast literature on the computation of low-speed viscous and inviscid fluid flows. Arguably the first contribution within this field is Chorin's algorithm \cite{Chorin}, who proposes to solve the incompressible Navier-Stokes equations using a projection method. Similar methods have also been used in, e.g., \cite{ColellaPao}.  A different approach to reduce the stiffness occurring at low Mach numbers is to introduce preconditioning, i.e., to multiply the temporal derivative with a suitable matrix, see the pioneering work by Turkel \cite{TurkelPreconditioning}, and, built on this result, the works by Guillard et al. \cite{Guillard2,Guillard1,Guillard3}. In \cite{Guillard1}, the author identifies the main problem in approximating low-speed flows: Roughly speaking, the variation of pressure is of second order in the Mach number $Ma$. However, 'traditional' Godunov schemes tend to produce pressure variations that are of first order in the Mach number, thus for $Ma \rightarrow 0$, there can be no 
uniform 
convergence. For an extensive additional analysis, in particular with respect to suitable initial conditions, we refer to Dellacherie \cite{Dellacherie}. 
We do not intend to give a fully exhaustive overview on this topic. For an overview on the treatment of low-speed flows, we refer to \cite{ChoiMerkle} and the references therein; a more recent survey was given in \cite{surveyAsymptotic}.

A class of algorithms that has found particular attention are the \emph{Asymptotic Preserving} schemes introduced by Jin \cite{Jin99}, built on work with Pareschi and Toscani \cite{Jin1998}. For an excellent review article, consult \cite{Jin2012}; we refer to
\cite{ArNo12,CoDeKu12,DegLoNaNe12,DeTa,HaJiLi12,ArNoLuMu12,JsAP13} for various applications of this method in different contexts.

Many algorithms, especially those used within the \emph{Asymptotic Preserving} schemes rely on
identifying \emph{stiff} and \emph{nonstiff} parts of the underlying equation. This point is generally considered crucial, and the hope is that a well-chosen splitting guarantees a good behavior of the algorithm. A splitting is usually obtained by physical reasoning, see, e.g., the fundamental work by Klein \cite{Kl95}.

Having obtained a splitting into stiff and nonstiff parts, the
nonstiff part is then treated explicitly, and the stiff one
implicitly. This procedure naturally leads to so-called IMEX
schemes as introduced in \cite{Ascher1997}. We refer to \cite{Bo07,RuBosc12} for an interesting discussion on the quality of these
schemes in the asymptotic limit.

As far as the authors can see, a fully nonlinear asymptotic stability analysis for the non-isentropic Euler equations is still out of reach. In this work, we attempt to reveal an important structural stability property of flux splittings via the considerably simpler \emph{modified equation} analysis \cite{WaHy74} for a prototype, $3\times3$ linear system of conservation laws.

More specifically we derive the modified parabolic system of equations of second order
and investigate under what conditions its solutions are bounded in the  $L^2$-norm.
For simple problems, one can investigate these conditions analytically. This approach is closely related to
the classical \emph{von-Neumann} stability analysis (see, e.g., \cite{Lax1961,Richtmyer1957}). Strang \cite{str64} showed that, under some assumptions, it is enough to consider only
linearized problems, so the approach used in this work is actually
more general than it seems at first sight.

Throughout the paper, we consider the {linear} hyperbolic system of conservation laws
 \begin{subequations}
 \label{eq:underlyingequation}
\begin{alignat}{2}
  u_t + A u_x &= 0 &\quad& \forall (x, t) \in \OO \times (0, T) \\
  u(x, 0) & = u_0(x) &\quad& \forall x \in \OO
\end{alignat}
 \end{subequations}
with a constant matrix $A \in \R^{d \times d}$ that has $d$
distinct eigenvalues and a full set of corresponding eigenvectors.
For simplicity, we set $\OO := [0,1]$ and consider smooth periodic
solutions $u$ with initial data $u_0$. Furthermore, we assume that the matrix $A$ is a function of a parameter $\eps \in (0, 1]$ such that with $\varepsilon
\rightarrow 0$, some eigenvalues of $A$ diverge towards infinity.

The motivation to consider \eqref{eq:underlyingequation} stems
from considering \emph{linearized} versions of classical systems
of conservation laws
\begin{align}
 \label{eq:conslaw}
 v_t + f(v)_x &= 0,
\end{align}
e.g., the (non-dimensionalized) Euler equations at low Mach number $\eps$, with characteristic quantities density, momentum and total energy,  $v = (\rho, \rho \vel, E)^T$, and
\begin{align}
 \label{eq:f_euler}
 f(\rho, \rho \vel, E) := \left(\rho \vel, \rho \vel^2 + \frac{p}{\eps^2}, \vel(E + p)\right)^T.
\end{align}
Its linearization around a state $(\rho, \rho \vel, E)$ yields a matrix $A$ with eigenvalues 
  \begin{align*}
   \lambda = \vel, \vel \pm \frac{c}{\eps}, 
  \end{align*}
  where $c := \sqrt{\frac{\gamma p}{\rho}}$ denotes the speed of sound. $\gamma$ is the ratio of specific heats, frequently taken to be $\gamma = 1.4$ for air. 
Note that two eigenvalues tend to infinity as $\eps \rightarrow 0$.

To highlight the difficulties posed by eigenvalues of multiple
scales we briefly discuss a standard, explicit finite volume scheme for \eqref{eq:conslaw}
\begin{alignat*}{2}
 \frac{v_j^{\itime+1}-v_j^\itime}{\dt} +
 \frac{\HHn_{j+\frac12}^\itime - \HHn_{j-\frac12}^\itime}{\dx} = 0
\end{alignat*}
with consistent numerical flux $\HHn_{j+\frac12}^\itime$.
From the stability conditions by Courant, Friedrichs and Lewy \cite{Courant1928}, it is known that explicit schemes are only stable under a $\cfl$ condition, which is typically given by 
\begin{align}
 \label{eq:cflmax}
 \cflmax := \lmax \dtdx = \frac{(\vel + \frac{c}{\eps}) \dt}{\dx}< 1. 
\end{align}
In the limit as $\eps \rightarrow 0$, one mainly wants to resolve the advective wave traveling with speed $\vel$. Given the restrictive $\cfl$ condition \eqref{eq:cflmax}, this would imply that one needs ${\mathcal O}(\eps^{-1})$ steps to advect a signal across a single grid cell. For small $\eps$, this is prohibitively inefficient, and for many schemes also prohibitively dissipative. However, using 
\begin{align}
 \label{eq:cfladv}
 \widehat \nu := \vel \dtdx < 1 
\end{align}
as \emph{advective} $\cfl$ condition would result in an unstable scheme. 

One potential remedy is to use fully implicit or mixed implicit /
explicit (IMEX) methods. The latter class of methods requires a
splitting of the flux $f$ into components with 'slow' and 'fast' waves.
More precisely, in the context of \eqref{eq:underlyingequation}, one seeks matrices $\widehat A$ and $\widetilde
A$, such that
\begin{align}
 \label{eq:splitting}
 A =  \widehat A + \widetilde A,
\end{align}
with the following conditions posed on $\widehat A$ and $\widetilde A$:
\begin{definition}\label{def:splitting}
 The splitting \eqref{eq:splitting} is called \emph{admissible}, if
 \begin{itemize}
  \item both $\widehat A$ and $\widetilde A$ induce a hyperbolic system, i.e.,
   they have real eigenvalues and a complete set of eigenvectors.
  \item the eigenvalues of $\widehat A$ are bounded independently of $\eps$.
 \end{itemize}

\end{definition}

In this work, we give a recipe for identifying stable classes of flux splittings for the linear hyperbolic system \eqref{eq:underlyingequation}. We use the well-known modified equation analysis as a tool for (heuristically) investigating (linear) $L^2$-stability.

The paper is outlined as follows: In Section 2, we introduce the lowest-order IMEX scheme, while in Section 3, we investigate this scheme using the modified equation approach. 
In Section 4, we introduce so-called characteristic splittings, which are a basic ingredient for a uniformly stable scheme. 
Based on those sections, we show our main result in Section 5, Theorem \ref{thm:charsplittingstable}: Characteristic splittings are stable in the sense as explained in Section 3 with a time step size \emph{independently} of $\eps$. 
Section 6 shows an example of a scheme that is only stable with a time step size decreasing with $\eps$, i.e., the allowable $\dt$ such that the scheme is stable behaves as $\dt \propto \eps$, which, for small $\eps$ is obviously not the desired effect. 
In Section 7, we apply our analysis to the linearized Euler equations and show some numerical results that substantiate the theory developed in this paper. 
Finally, Section 8 offers conclusions and possible future work.


\section{IMEX Discretization}

Based on a splitting as given in \eqref{eq:splitting}, we
introduce a straightforward first-order IMEX discretization for \eqref{eq:underlyingequation} based
on nonstiff and stiff numerical fluxes $\HHn$ and $\HHs$. We
assume that the temporal domain is subdivided as
\begin{align*}
 0 = t^0 < t^1 < \ldots < t^{\imaxtime} = T
\end{align*}
with constant spacing $\dt := t^{\itime+1}-t^\itime$; and that we have a subdivision
\begin{align*}
 \OO := \bigcup_{\ispace=0}^{\imaxspace} [x_{\ispace-\frac12}, x_{\ispace+\frac12}]
\end{align*}
also with constant spacing $\dx := x_{\ispace+\frac12} - x_{\ispace-\frac12}$
and cell midpoints $x_\ispace$.
As is customary, we denote a numerical approximation to $u(x_\ispace,
t^\itime)$ by $u_\ispace^\itime$. Furthermore, the vector
$(u_0^\itime, \ldots,  u_{\imaxspace}^{\itime})$ is denoted by
$\vec{U}^\itime$.
Now we can introduce a (classical) first-order IMEX scheme:
\begin{definition}
A sequence $\vec{U} = (\vec{U}^0, \ldots, \vec{U}^{\imaxtime})$ is a solution to an IMEX discretization, given that
\begin{align*}
 \quad \IMEX(\vec{U}) = 0, \forall j \in \{0, \ldots, J\}, \ \forall n \in \{0, \ldots, N-1\}, 
\end{align*}
where
\begin{alignat}{2}
 \label{eq:imex}
 \IMEX(\vec{U}) := \frac{u_j^{\itime+1}-u_j^\itime}{\dt}
  + \frac{\HHn_{j+\frac12}^\itime - \HHn_{j-\frac12}^\itime}{\dx} + \frac{\HHs_{j+\frac12}^{\itime+1} - \HHs_{j-\frac12}^{\itime+1}}{\dx}.
\end{alignat}
Here, nonstiff and stiff numerical fluxes are defined by
\begin{alignat}{2}
 \label{eq:HHn}
  \HHn_{j+\frac12}^\itime     & := \frac12 \widehat   A \left( u^\itime_{j+1}+u^\itime_j \right)         &-& \frac{\widehat \alpha}{2} \left(u^\itime_{j+1}-u^\itime_j \right)
 \\ \label{eq:HHs}
  \HHs_{j+\frac12}^{\itime+1} & := \frac12 \widetilde A \left( u^{\itime+1}_{j+1}+u^{\itime+1}_j \right) &-& \frac{\widetilde \alpha}{2} \left(u^{\itime+1}_{j+1}-u^{\itime+1}_j \right),
\end{alignat}
with (positive) numerical viscosities $\widehat\alpha$ and $\widetilde \alpha$.
\end{definition}

\begin{remark}
 $\HHn$ and $\HHs$ are given in the so-called viscosity form of a numerical flux, see, e.g., \cite{GR1}. More generally, one can also consider matrices for $\widehat \alpha$ and $\widetilde \alpha$ instead of scalars. This plays a role in preconditioned schemes; here, it is omitted for the sake of simplicity. 
\end{remark}

For fixed $\eps$, consistency analysis of the scheme is well-known, see, e.g., \cite{Crouzeix1980,GR1,Kroener}. However, we have to consider both $\eps$ and $\dt$ as small parameters. The crucial point is that we restrict our analysis to cases where the magnitude of $u$ and its first and second derivatives are independent of $\eps$. (Especially, no derivative behaves as $O(\eps^{-1})$ or worse.) This assumption is reasonable, as only those solutions allow for an asymptotic limit as $\eps \rightarrow 0$. Similar assumptions have been made in \cite{KlMa81}.
\begin{lemma}
 Let $\vec{\underline U}^\itime$ denote the vector $\left(u(x_0, t^\itime), \ldots, u(x_{\Nx}, t^\itime)\right)$ with $u$ solution to \eqref{eq:underlyingequation} whose derivatives can be bounded \emph{independently} of $\eps$;  and let $\vec{\underline U} := (\vec{\underline U}^0, \ldots, \vec{\underline U}^{\imaxtime})$. Furthermore, let $O(\dx) = O(\dt)$. Then, the local truncation error is of order $\Delta x$, i.e., there holds for all $j = \{0, \ldots, J\}$ and for all $n \in \{0, \ldots N-1\}$, 
 \begin{align*}
    \IMEX(\vec{\underline U}) = O(\Delta x).
 \end{align*}
 \end{lemma}
 \begin{proof}
  Apply a Taylor expansion to \eqref{eq:imex} and note that all the derivatives of $u$ can be bounded independently of $\eps$. \qed
 \end{proof}
 
 \begin{remark}
  Note that the condition that the derivatives of $u$ can be bounded \emph{independently} of $\eps$ is indeed a condition on the initial datum $u_0$. 
  Not every initial data gives rise to such a solution. 
  More precisely, this means that as $\eps \rightarrow 0$, there is indeed a solution to \eqref{eq:underlyingequation} that does not blow up. In the context of the low-Froude or low-Mach number limit, such a choice of initial conditions is often called \emph{well-prepared initial data}. For a discussion of the influence of initial conditions on the limit solution, we refer to the work by Klainerman and Majda \cite{KlMa81}. 
  However, also for examples from ordinary differential equations, there are analogues, see, e.g., \cite{Bo07,HaiWan}. 
 \end{remark}

\section{Modified Equation Analysis}
In this section, we derive the modified equation \cite{WaHy74}
corresponding to \eqref{eq:imex}. As we consider a periodic
setting, we can solve the resulting parabolic system explicitly
using Fourier series. Using Plancherel's theorem, we investigate
the stability of the modified equation. This yields a practical
criterion for the stability of the IMEX scheme.

We start by deriving the modified equation corresponding to \eqref{eq:imex}. 
\begin{theorem}\label{thm:modeq}
 Let $w$ be a smooth solution of
 \begin{align}
  \label{eq:mod_eq}
    w_t + Aw_x &= \frac{\dt}{2}
     \left(\frac{\left(\widehat \alpha
      + \widetilde \alpha\right)\dx}{\dt} \Id -  (\Ahat-\Atilde)A \right)
      w_{xx}.
 \end{align}
 Furthermore, we consider vectors $\vec{\underline W}^n := \left(w(x_0,t^n), \ldots, w(x_{\imaxspace}, t^n) \right)$ and $\vec{\underline W} := (\vec{\underline W}^0, \ldots, \vec{\underline W}^{\imaxtime})$.
 Then, for fixed $\eps$ and $O(\dx) = O(\dt)$, the IMEX scheme \eqref{eq:imex} is a second order accurate discretization of
 \eqref{eq:mod_eq}, i.e.
 \begin{align*}
  \IMEX(\vec{\underline W}) = O(\Delta x^2).
 \end{align*}
 \end{theorem}
 \begin{proof}
 It is well-known that the modified equation for a first-order discretization is a parabolic equation, i.e., we expect $w$ to fulfill
 \begin{align}
  \label{eq:visc_prototype}
  w_t + A w_x &= B w_{xx}
 \end{align}
 for a (yet unknown) viscosity matrix $B$ that is in class $O(\dx)$. 
 Using \eqref{eq:visc_prototype}, one can conclude that 
 \begin{align}
  \label{eq:wt}
  w_t    &= -A w_x + B w_{xx} \\
  \label{eq:wtt}
  w_{tt} &= -A (w_t)_x + B (w_t)_{xx} \stackrel{\eqref{eq:wt}}= A^2 w_{xx} + O(\dx).
 \end{align}
 Note that this holds due to $O(\dx) = O(\dt)$, which we will from now on exploit frequently. 
 To simplify the presentation, we slightly abuse our notation, and write $\ww_\ispace^\itime$ for $w(x_\ispace, t^\itime)$.
 Using \eqref{eq:wtt} at position $(x_\ispace, t^\itime)$,
 \begin{align}
  \notag
  \frac{\ww_\ispace^{\itime+1}-\ww_\ispace^\itime}{\dt}
   &= w_t + \frac{\dt}{2} w_{tt} + O(\dx^2) \\
   &= w_t + \frac{\dt}{2} A^2 w_{xx} + O(\dx^2) \label{eq:apprt2}
 \end{align}
 and
 \begin{align}
 \notag
  \frac{\HHn_{j+\frac12}^\itime - \HHn_{j-\frac12}^\itime}{\dx}
   &= \frac1{2\dx} \Ahat \left(\ww_{j+1}^\itime - \ww_{j-1}^\itime \right)
    - \frac{\widehat \alpha}{2\dx}
     \left(\ww_{\ispace-1}^\itime - 2 \ww_\ispace^\itime + \ww_{\ispace+1}^\itime \right)
   \\ &=
    \Ahat \, w_x - \frac{\widehat \alpha \dx}{2} w_{xx} + O(\dx^2).\label{eq:explf2}
 \end{align}
 Similarly,
 \begin{align*}
   \frac{\HHs_{j+\frac12}^{\itime+1} - \HHs_{j-\frac12}^{\itime+1}}{\dx}
   &=
   \frac1{2\dx} \Atilde
    \left(\ww_{j+1}^{\itime+1} - \ww_{j-1}^{\itime+1} \right)
   - \frac{\widetilde \alpha}{2\dx}
    \left(\ww_{\ispace-1}^{\itime+1} - 2 \ww_\ispace^{\itime+1} + \ww_{\ispace-1}^{\itime+1} \right)
  \\ &=
   \Atilde w_x(x_\ispace,t^{\itime+1})
   - \frac{\widetilde \alpha \dx}{2} w_{xx}(x_\ispace,t^{\itime+1})
   + O(\dx^2).
 \end{align*}
 From \eqref{eq:wt},
 \begin{align*}
   w_x(x_\ispace,t^{\itime+1})
  &=
  w_x(x_\ispace,t^{\itime}) - \dt A w_{xx} + O(\dx^2),
 \end{align*}
while $w_{xx}(x_\ispace,t^{\itime+1}) =
w_{xx}(x_\ispace,t^{\itime}) + O(\dx)$.
 Therefore,
 \begin{align}
 \label{eq:implf}
   \frac{\HHs_{j+\frac12}^{\itime+1} - \HHs_{j-\frac12}^{\itime+1}}{\dx}
  &=
  \Atilde w_x - \dt \Atilde A w_{xx}
   - \frac{\widetilde \alpha \dx}{2} w_{xx} + O(\dx^2).
 \end{align}

Now we plug \eqref{eq:apprt2}, \eqref{eq:explf2} and
\eqref{eq:implf} into \eqref{eq:imex} to obtain, always at position
$(x_\ispace, t^\itime)$,
 \begin{align*}
   & \hphantom{=}\IMEX(\vec{\underline W})
  \\ &=
   w_t + \frac{\dt}{2} A^2 w_{xx}
    + \widehat A w_x - \frac{\widehat \alpha}{2} \dx w_{xx}
  \\ &\hphantom{=}
   + \widetilde A w_x - \dt \widetilde A A w_{xx}
    - \frac{\widetilde \alpha}{2} \dx w_{xx} + O(\dx^2 )
  \\ &=
   w_t + (\widehat A + \widetilde A) w_x
  \\ &\hphantom{=}
   + \frac{\dt}{2} \left(A^2  - 2\widetilde A A
    - {\widehat \alpha} \frac{\dx}{\dt} \Id
     - {\widetilde \alpha} \frac{\dx}{\dt}\Id\right) w_{xx} + O(\dx^2).
 \end{align*}
 This is $O(\dx^2)$ if and only if $w$ fulfills \eqref{eq:visc_prototype} with
 \begin{align}
  \notag
  B &= \frac{\dt}{2} \left(-A^2  + 2\widetilde A A + {\widehat \alpha} \frac{\dx}{\dt} \Id + {\widetilde \alpha} \frac{\dx}{\dt} \Id \right) \\
  \label{eq:BB}
    &= \frac{\dt}{2} \left(\frac{\left(\widehat \alpha + \widetilde \alpha \right) \dx}{\dt} \Id - (\widehat A - \widetilde A) A \right) .
 \end{align}
 Note again that we have repeatedly used the assumption $O(\dx) = O(\dt)$. This proves the claim. \qed
\end{proof}

In the sequel, we show how \eqref{eq:mod_eq} can be used to
determine a necessary condition under what $\cfl$ condition the IMEX
scheme \eqref{eq:imex} is stable. We begin by deriving an
exact solution to \eqref{eq:visc_prototype} using a Fourier
ansatz. Note that $A$ is a $d\times d$
matrix.
\begin{lemma}\label{la:parabolic}
Let $w_0$ be given by
 \begin{align}
 \label{eq:w0fourier}
  w_0(x) &= \sum_{\ifourier \in \Z} \vecd{a^1_{0\ifourier}}{a^d_{0\ifourier}} e^{i 2 \pi \ifourier x}.
 \end{align}
 Furthermore, let $w$ be a solution to
 \begin{alignat}{2}
  \label{eq:parabolic}
  w_t + A w_x &= Bw_{xx} &\quad& \forall (x, t) \in \OO \times (0, T) \\
  w(x, 0) & = w_0(x) &\quad& \forall x \in \OO. \notag
 \end{alignat}
 Then, $w$ admits a representation
 \begin{align}
 \label{eq:wfourier}
  w(x, t) &= \sum_{\ifourier \in \Z} \vecd{a^1_\ifourier(t)}{a^d_\ifourier(t)} e^{i 2 \pi \ifourier x}
 \end{align}
 with $a^1_\ifourier, \ldots, a^d_\ifourier$ fulfilling the system of $d$ ordinary differential equations
 \begin{align}
 \label{eq:anbn}
  \vecd{a^1_\ifourier(t)'}{a^d_\ifourier(t)'} &= \Aa \vecd{a^1_\ifourier(t)}{a^d_\ifourier(t)}
 \end{align}
 for
 \begin{align}
  \label{eq:AA}
  \Aa := (-i 2 \pi \ifourier A - 4 \pi^2 \ifourier^2 B)
 \end{align}
 and initial conditions
 \begin{align*}
  \vecd{a_\ifourier^1(0)}{a_\ifourier^d(0)} &= \vecd{a^1_{0\ifourier}}{a^d_{0\ifourier}}.
 \end{align*}
 \end{lemma}
 \begin{proof}
  The proof exploits direct computations and starts with \emph{assuming} that the representation \eqref{eq:wfourier} is correct. Thus, plugging \eqref{eq:wfourier} into \eqref{eq:parabolic}, one obtains
  \begin{align*}
   \sum_{\ifourier \in \Z} \left( \vecd{a^1_\ifourier(t)'}{a^d_\ifourier(t)'}  + i 2 \pi \ifourier  A \vecd{a^1_\ifourier(t)}{a^d_\ifourier(t)}  +  4 \pi^2 \ifourier^2 B \vecd{a^1_\ifourier(t)}{a^d_\ifourier(t)}  \right) e^{i 2 \pi \ifourier x} &= 0.
  \end{align*}
  Exploiting the linear independence of $e^{i 2 \pi \ifourier x}$ for different $\ifourier$, one obtains \eqref{eq:anbn}. \qed
 \end{proof}

\begin{remark}
\begin{itemize}
 \item Every periodic smooth function $w_0$ can be written as in \eqref{eq:w0fourier}. 
 \item For future reference, we call $\Aa$ the {\em frequency matrices} of the modified equation  \eqref{eq:mod_eq}.
\end{itemize}
\end{remark}

The following corollary is a direct consequence from the theory of ordinary differential equations, and Plancherel's theorem.
\begin{corollary}\label{cor:l2norm}
 We consider the setting as in Lemma \ref{la:parabolic}. Then,
 \begin{align}
  \label{eq:l2normconservation}
  \|w(\cdot, t)\|_{L^2(\OO)} \leq C \|w_0(\cdot)\|_{L^2(\OO)}
 \end{align}
 holds for a positive constant $C$ if
 \begin{align*}
  \Real(\eigAk_{k,i}) < 0
 \end{align*}
 for all eigenvalues $\eigAk_{k,i}$ of $\Aa$ with
 $\ifourier \in \Z^{\neq 0}$. 
\end{corollary}

\begin{remark}
 One might argue that Corollary \ref{cor:l2norm} is not needed in the sense that for every matrix $B$ with $B$ positive definite, there holds \eqref{eq:l2normconservation}. 
 However, this is not a necessary condition. Consider, e.g., the pair of matrices $A = \Id$ and $B = \left(\begin{matrix} 5 & 1 \\ -2 & 0 \end{matrix} \right)$. Obviously, $B$ is not positive definite (note that $x^T B x < 0$ for, e.g., $x := (1, 10)^T$), however, the eigenvalues of $\Aa$ have negative real part, and consequently, the complete system \eqref{eq:parabolic} is stable. (A tedious computation reveals that the eigenvalues of $\Aa$ are $2\pi k\left((\pm\sqrt{17} - 5)\pi k - i \right)$.)
 \end{remark}
 
 As already pointed out in the introduction, we have the following important remark concerning commutative matrices:
\begin{remark}
The real part of the eigenvalues of $\Aa$ is \emph{not} affected by the terms stemming from $A$ if matrices $A$ and $B$ can be simultaneously diagonalized. This is the motivation for introducing so called \emph{characteristic splittings} in the following section.
 \end{remark}

\section{Characteristic Splitting}
In this section, we introduce a new class of splittings that, with our analysis to be presented, turns out to be uniformly stable in $\eps$ \emph{without} any additional stabilization terms. The splitting relies on a characteristic decomposition of the matrix $A$, i.e., $A$ can be decomposed into 
  \begin{align}
   \label{eq:charsplittingA}
   A = Q \Lambda Q^{-1}
  \end{align}
  for an invertible $Q$ and $\Lambda := \diag(\lambda_1,\ldots,\lambda_d)$. 
  The idea of the characteristic splitting is to split the matrix $\Lambda$ into stiff and nonstiff parts. We make this more precise in the following definition:
  \begin{definition}
  Let $A$ be decomposed as in \eqref{eq:charsplittingA}, and let $\Lambda$ be split into 
  \begin{align*}
   \Lambda &= \widehat \Lambda + \widetilde \Lambda
  \end{align*}
  in such a way that $\widehat \Lambda$ and $\widetilde \Lambda$ are diagonal matrices and define an admissible splitting of $\Lambda$ in the sense of Definition \ref{def:splitting}. 
  (Consequently, the entries of $\widehat \Lambda$ can be bounded independently of $\eps$.) Subsequently, the \emph{characteristic splitting} is defined by
  \begin{align}
   \label{eq:splitting3dA}
   \Ahat = Q \widehat \Lambda Q^{-1}
   \quad \mathrm{and} \quad
   \Atilde = Q \widetilde \Lambda Q^{-1}.
  \end{align}   
  \end{definition} 

  Obviously, the splitting of $A$ is admissible in
  the sense of Definition \ref{def:splitting}.
 
 Let us make the following remark about characteristic splittings:
 \begin{remark}
  \begin{itemize}
   \item As the system \eqref{eq:underlyingequation} is hyperbolic for all $\eps > 0$, one can always obtain an admissible characteristic splitting by choosing $\widehat \Lambda$ as $\Lambda_{|\eps = 1}$, and $\widetilde \Lambda := \Lambda - \widehat \Lambda$. Obviously, the eigenvalues of both $\widehat A$ and $\widetilde A$ are real, and those of $\widehat A$ are trivially independent of $\eps$. Such a decomposition would lead to a fully explicit scheme for $\eps = 1$, which is desirable as for nonstiff equations, those schemes usually are less diffusive than implicit ones. 
   \item Note that $Q$ still depends on $\eps$, and so, even if $\widehat \Lambda$ is independent of $\eps$, $\widehat A$ generally is not (but its eigenvalues are). 
   \item As one can see from Section \ref{sec:stability} and especially Section \ref{sec:nonuniformity}, a crucial part in our analysis is the fact that $\widehat A$ and $\widetilde A$ commute, i.e., $\widetilde A \widehat A = \widehat A \widetilde A$. In this way, the 'bad' modes that can potentially destroy uniform stability are ruled out. 
  \end{itemize}

 \end{remark}

 In the sequel, we apply this concept to a prototype matrix $A$, given by
 
\label{sec:prototypesplitting}
  %
  \begin{align}
  \label{eq:3dA}
  A = \left( \begin{matrix} a & 1 & 0 \\ \frac1{\eps^2} & a & \frac1{\eps^2} \\ 0 & 1 & a \end{matrix} \right).
  \end{align}
  Its eigenvalues are
  \begin{align*}
   \lambda = a, a \pm \frac{\sqrt{2}}{\eps},
  \end{align*}
and for simplicity, we consider $a$ to be positive, i.e., $\lmax
:= a + \frac{\sqrt{2}}{\eps}$ is the largest eigenvalue.

In order to be fully explicit for $\eps = 1$, we use a characteristic splitting with 
  \begin{subequations}
  \begin{align*}
   \widehat\Lambda &:= \diag \left(a-\sqrt{2}, a, a + \sqrt{2}\right),
   \\
   \widetilde \Lambda &:= \diag \left(-\frac{\sqrt{2}(1-\eps)}{\eps},
    0, \frac{\sqrt{2}(1-\eps)}{\eps}\right).
  \end{align*}
  \end{subequations}
Consequently, we can derive matrices $\widehat A$ and $\widetilde A$ via \eqref{eq:splitting3dA} as
  \begin{align*}
   \widehat   A
    =  \left(\begin{matrix} a & \eps & 0 \\
     \frac1{\eps} & a & \frac{1}{\eps} \\
      0 & \eps & a \end{matrix} \right),
    \quad
   \widetilde A
    =  \left(\begin{matrix} 0 & 1-\eps & 0 \\
     \frac{1-\eps}{\eps^2} & 0 & \frac{1-\eps}{\eps^2} \\
      0 & 1-\eps & 0\end{matrix} \right).
  \end{align*}

The focus of this paper is on uniform stability  as $\eps \to 0$,
where the fast wave speeds tend to infinity. As outlined in the
introduction, the goal is to overcome the inefficiency of a fully
explicit scheme due to condition \eqref{eq:cflmax}, or the instability due to condition
\eqref{eq:cfladv}. 
In the following (cf. Theorem~\ref{thm:charsplittingstable}), we derive
upper bounds on the nonstiff $\CFL$ number that assure stability (in a sense  to be made more precise) of
IMEX scheme \eqref{eq:imex} for a characteristic splitting.

\section{Stability Of Characteristic Flux Splittings}\label{sec:stability}
Now, we combine Theorem \ref{thm:modeq} and Corollary
\ref{cor:l2norm} to obtain a necessary criterion under what
circumstances the IMEX scheme \eqref{eq:imex} is stable. 
We start with the general case, and subsequently consider the prototype equation. 

\subsection{General case}

 We consider the characteristic splitting \eqref{eq:splitting3dA} in the light of Corollary \ref{cor:l2norm}. For a generic splitting with \emph{commuting} matrices $\widehat A$ and $\widetilde A$, the frequency matrix $\Aa$ (see \eqref{eq:AA} and \eqref{eq:BB}) can be written as 
\begin{align}
  \label{eq:aa}
  \Aa &= -i 2 \pi \ifourier A - 2 \pi^2 \ifourier^2 {\dt} \left(\frac{\left(\widehat \alpha + \widetilde \alpha \right) \dx}{\dt} \Id - (\widehat A - \widetilde A) (\widehat A + \widetilde A) \right) \\
  \label{eq:aa2}
  &= -i 2 \pi \ifourier A - 2 \pi^2 \ifourier^2 {\dt} \left(\frac{\left(\widehat \alpha + \widetilde \alpha \right) \dx}{\dt} \Id - \widehat A^2 + \widetilde A^2 \right). 
\end{align}
Note that $\mathcal{A}_0 = 0$, since constant (in space) solutions of the
modified equations are constant in time also. Therefore we need to
analyze only the case $k\neq0$.
As we rely on a characteristic splitting, all the matrices occurring in \eqref{eq:aa2} can be written as $Q \Sigma Q^{-1}$ for some diagonal matrix $\Sigma$. Thus, it is easy to see that the real part $\eigAk_{k,i}$ of the eigenvalues of $\Aa$ is given by 
 \begin{align*}
  \Real(\eigAk_{k,i}) = 2\pi^2 k^2 \dt \left(-\frac{\left(\widehat \alpha + \widetilde \alpha \right) \dx}{\dt} + \widehat \lambda_i^2 - \widetilde \lambda_i^2 \right)
 \end{align*}
 where $\widehat \lambda_i$ and $\widetilde \lambda_i$ are eigenvalues to $\widehat A$ and $\widetilde A$, respectively. Claiming that $\Real(\eigAk_{k,i})$ is negative leads to 
 \begin{align}
  \label{eq:splittingrestriction1}
  \frac{\left(\widehat \alpha + \widetilde \alpha \right) \dx}{\dt} > \widehat \lambda_i^2 - \widetilde \lambda_i^2. 
 \end{align}
 This leads to a time step restriction that only depends on the \emph{explicit} part. We summarize this in the following lemma: 
 \begin{lemma} 
  Let $\frac{\dt}{\dx}$ be restricted by 
    \begin{align}
    \label{eq:splittingrestriction}
    \frac{\dt}{\dx} < \frac{\widehat \alpha + \widetilde \alpha}{\max_i \widehat \lambda_i^2}
    \end{align}
  Then, \eqref{eq:splittingrestriction1} and thus Corollary \ref{cor:l2norm} hold.
 \end{lemma}
 
 %
 This is a good result, as $\widehat \lambda_i^2$ can be bounded independently of $\eps$, and therefore, \eqref{eq:splittingrestriction} is also independent of $\eps$. We summarize this in the following theorem.

\begin{theorem}\label{thm:charsplittingstable}
 The characteristic splitting as introduced in \eqref{eq:splitting3dA} is such that $\Real(\eigAk_{k,i}) < 0$ holds for all $\ifourier \in \Z^{\neq 0}$, $\eigAk_{k,i}$ eigenvalue to $\Aa$, with a restriction on the time step size that is independent of $\eps$. 
\end{theorem}

In the sequel, we consider a prototype system in more detail to obtain quantitative information. 

\subsection{Characteristic Splitting of prototype equation}
In this section, we consider the prototype matrix from Section \ref{sec:prototypesplitting}, as it allows for easy and explicit computations.
The non-dimensionalized advective $\cfl$ number corresponding to $A$ is denoted by 
\begin{align}
 \label{eq:widehatnu}
 \widehat \nu := \frac{a \dt}{\dx}.
\end{align}
Note that using $\widehat \nu$, \eqref{eq:splittingrestriction} reads
\begin{align*}
 \widehat \nu < \frac{a\left(\widehat \alpha + \widetilde \alpha\right)}{(a + \sqrt{2})^2}.
\end{align*}
In the sequel, we give stronger bounds on the allowable time step size. 
  Given $k \in \Z^{\neq 0}$, we consider the frequency matrix $\Aa$ for the
  characteristic splitting introduced in \eqref{eq:splitting3dA}.
  One potential advantage of the characteristic splitting is that one can compute the eigenvalues explicitly, as all the matrices commute: 
  \begin{lemma} 
    The real part of the eigenvalues $\{\eigAk_{k,0}, \eigAk_{k,\pm}\}$ of $\Aa$ are given by 

  \begin{subequations}
  \begin{align}
  \label{eq:3dl1}
  \Real(\eigAk_{k,0})     &= -2\pi^2 k^2 \dx (\widehat \alpha + \widetilde \alpha) + 2 \pi^2 k^2 \dt a^2 \\
  \Real(\eigAk_{k,\pm}) &= \frac{-4 \pi^2 k^2 \dt}{\eps^2} + \frac{8 \dt k^2 \pi^2}{\eps}  \label{eq:re_l22}\\
		       & \hphantom{= } + 2 \pi^2 k^2 a^2 \dt - 2 \pi^2 k^2 \dx(\widehat \alpha + \widetilde \alpha) \pm 4 \sqrt{2} a \pi^2 k^2 \dt. \notag
  \end{align}
  \end{subequations}
  \end{lemma}
  \begin{proof}
     With the notation introduced in \eqref{eq:charsplittingA} and \eqref{eq:splitting3dA}, see also \eqref{eq:aa2}, there holds
     \begin{align*}
      \mathcal{A}_k & = -i2\pi k Q \Lambda Q^{-1} -2\pi^2 k^2 \dt \left(\frac{\left(\widehat \alpha + \widetilde \alpha\right) \dx}{\dt} \Id - Q (\widehat \Lambda - \widetilde \Lambda) \Lambda Q^{-1}   \right)  \\
		    & = Q \left(-i2\pi k \Lambda  -2\pi^2 k^2 \dt \left(\frac{\left(\widehat \alpha + \widetilde \alpha\right) \dx}{\dt} \Id - (\widehat \Lambda - \widetilde \Lambda ) \Lambda  \right) \right) Q^{-1} \\
		    & = Q \diag(\sigma) Q^{-1} ,
     \end{align*}
     where the vector $\sigma$ is given by 
     \begin{align*}
      \sigma := \left( 
		      \begin{matrix} 
			-i 2 \pi k (a - \frac{\sqrt{2}}{\eps}) - 2 \pi^2 k^2 \dt \left(\frac{\left(\widehat \alpha + \widetilde \alpha\right) \dx}{\dt} + \frac{2}{\eps^2} - \frac{4}{\eps} +2 \sqrt{2} a - a^2 \right)  \\
			-i 2 \pi k a - 2 \pi^2 k^2 \dt \left(\frac{\left(\widehat \alpha + \widetilde \alpha\right) \dx}{\dt} - a^2\right)  \\
			-i 2 \pi k (a + \frac{\sqrt{2}}{\eps}) - 2 \pi^2 k^2 \dt \left(\frac{\left(\widehat \alpha + \widetilde \alpha\right) \dx}{\dt} + \frac{2}{\eps^2} - \frac{4}{\eps} -2 \sqrt{2} a - a^2 \right)  \\
		      \end{matrix}\right) 
     \end{align*}
     Thus, one can conclude that the eigenvalues of $\mathcal{A}_k$ are given in $\sigma$. Sorting the eigenvalues conveniently, starting with the one in the middle, one can conclude that their Real parts are given by formulae \eqref{eq:3dl1} and \eqref{eq:re_l22}.\qed
    \end{proof}

 The problem under consideration has two asymptotics, namely the one associated to
 $\eps \rightarrow 0$, and the other one associated to $\dt \rightarrow 0$
 (which automatically includes $\dx \rightarrow 0$).
 We immediately obtain the following condition for the negativity of the first eigenvalue:
 \begin{lemma}\label{lemma:lambda1}
  $\Real(\eigAk_{k,0}) < 0$  under the $\cfl$ condition
  \begin{align}
    \cflhat < \cfli := \frac{\widehat \alpha + \widetilde \alpha}{a}.
    \label{eq:cfli}
  \end{align}
 \end{lemma}

 \begin{remark}
  Condition \eqref{eq:cfli} is a condition for the nonstiff $\cfl$ number
  $\cflhat$ from \eqref{eq:widehatnu}. It depends on both $\widehat \alpha$ and $\widetilde \alpha$.
  The coefficient $\widehat \alpha$ encodes the upwind viscosity
  of the explicit numerical flux \eqref{eq:HHn}, and is usually chosen as
  $a + \sqrt{2}$, the largest eigenvalue of the nonstiff matrix.
  There is more freedom to choose the viscosity  coefficient of the implicit
  numerical flux \eqref{eq:HHs}, and limiting choices are either
  $\widetilde  \alpha = \frac{\sqrt{2} (1-\eps)}{\eps}$ (the largest eigenvalue
  of the nonstiff matrix) or $\widetilde \alpha = 0$.
  In both cases, $\widehat \alpha + \widetilde \alpha
  \geq a+\sqrt2$, which gives the sufficient stability
  condition
  \begin{align*}
    \cflhat < \frac{a+\sqrt2}{a}.
  \end{align*}
  This is independent of $\eps$.
 \end{remark}

 Now we discuss $\Real(\eigAk_{k,\pm})$. Obviously, for $\dt$ fixed and $\eps
\rightarrow 0$, $\Real(\eigAk_{k, \pm}) < 0$, which directly yields the following
Proposition:

 \begin{proposition}
  Let $\dt$ be fixed. Then, there exists an $\eps_0 > 0$, such that for all
  $\eps < \eps_0$ and all $\ifourier \in \Z^{\neq 0}$, $\Real(\eigAk_{k,\pm}) < 0$.
 \end{proposition}

 However, this is not the full asymptotics. We therefore change the point of view:
 Given a \emph{fixed} $\eps$, for which $\dt$ is  $\Real(\eigAk_{k,\pm}) < 0$?
 The following lemma provides the crucial estimate:
 
  \begin{lemma}\label{lemma:lambda23}
  We define 
  \begin{align}
   \label{eq:varphi}
   \varphi(a) := {\frac {\sqrt{2}}{a+2\,\sqrt {2} }}. 
  \end{align}
  Now, consider two cases as follows:
  \begin{enumerate}
   \item Let $\eps \leq \varphi(a)$. Then, $\Real(\eigAk_{k,\pm}) < 0$ holds unconditionally. 
   \item Let $\varphi(a) < \eps$. Then, $\Real(\eigAk_{k,\pm}) < 0$ holds for 
   \begin{align*}
    \widehat \nu \leq \frac{(\widehat \alpha + \widetilde \alpha)a}{(a + \sqrt{2})^2}. 
   \end{align*}
  \end{enumerate}
  \end{lemma}
 \begin{proof}
  For $\Real(\mu_{k,\pm})$ to be negative, it suffices to show that 
  \begin{align*}
   0 > \frac{-4 \dt}{\eps^2} + \frac{8 \dt }{\eps} + 2 a^2 \dt - 2 \dx(\widehat \alpha + \widetilde \alpha) + 4 \sqrt{2} a \dt.
  \end{align*}
  We substitute $\dx = \frac{a \dt}{\widehat \nu}$ and obtain 
  \begin{align}
   0 &> \frac{-4 }{\eps^2} + \frac{8}{\eps} + 2 a^2 - \frac{2 (\widehat \alpha + \widetilde \alpha)a}{\widehat \nu} + 4 \sqrt{2} a \notag \\
   \Leftrightarrow \frac{(\widehat \alpha + \widetilde \alpha)a}{\widehat \nu} &> \frac{-2}{\eps^2} + \frac{4}{\eps} + a^2  + 2 \sqrt{2} a. \label{eq:proof_ineq}
  \end{align}
  \eqref{eq:proof_ineq} is trivially fulfilled, if the right-hand side is not positive, i.e., if
  \begin{align*}
   0 &\geq \frac{-2}{\eps^2} + \frac{4}{\eps} + a^2  + 2 \sqrt{2} a \\ 
   \Leftrightarrow 0 &\geq -2 + 4 \eps + \eps^2 \left(a^2  + 2 \sqrt{2} a \right) \\
   \Leftrightarrow 0 &\leq \eps \leq \varphi(a). 
  \end{align*}
  This proves the first claim that for all $\eps \leq \varphi(a)$, both $\Real(\eigAk_{k,\pm})$ are negative. 
  
  Now let $\varphi(a) < \eps$. In this case, the right-hand side of \eqref{eq:proof_ineq} is positive, and therefore one has a restriction on $\widehat \nu$. One can compute 
  \begin{align}
		      \frac{(\widehat \alpha + \widetilde \alpha)a}{\widehat \nu} &> \frac{-2}{\eps^2} + \frac{4}{\eps} + a^2  + 2 \sqrt{2} a \notag \\
   \Leftrightarrow	\frac{\widehat \nu}{(\widehat \alpha + \widetilde \alpha) a} &< \frac{1}{\frac{-2 + 4\eps}{\eps^2} + a^2  + 2 \sqrt{2} a} \label{eq:proof_ineq2}. 
  \end{align}
  Note that for any $0 \leq \eps \leq 1$, there holds $\frac{-2+4\eps}{\eps^2} \leq 2$. Consequently, \eqref{eq:proof_ineq2} is fulfilled if 
  \begin{align*}
   \frac{\widehat \nu}{(\widehat \alpha + \widetilde \alpha)a} \leq \frac{1}{2 +  a^2  + 2 \sqrt{2} a} = \frac{1}{(a + \sqrt{2})^2}
  \end{align*}
  This proves the lemma.   \qed
 \end{proof}

With Lemmas \ref{lemma:lambda1} and
\ref{lemma:lambda23} we obtain the following

 \begin{proposition}\label{theorem:summary}
  Recall the definition of $\cfli$ from \eqref{eq:cfli}. 
  For $\ifourier\in \Z^{\neq 0}$, $\Real(\eigAk_{k,0})<0$ and $\Real(\eigAk_{k,\pm}) < 0$ if
  \begin{align}
   \label{eq:cflii}
   \cflhat < \cflii := \cfli \, \psi(\eps,a)
  \end{align}
  with
  \begin{align}
   \label{eq:psi}
   \psi(\eps,a) :=
     \begin{cases} 
      1, &\quad \eps \leq \varphi(a) \\
      \left( \frac{a}{a+\sqrt{2}} \right)^2 &\quad \eps > \varphi(a).
     \end{cases}
  \end{align}
  \end{proposition}

\bigskip

From the previous considerations, we can conclude:

\begin{corollary}
  We choose $\widehat \alpha = a + \sqrt{2}$. 
Then, there holds $\Real(\eigAk_{k,i}) < 0$ for all $\ifourier \in \Z^{\neq 0}$, $\eigAk_{k,i}$ eigenvalue to $\Aa$, if 
  \begin{align}
  \label{eq:regcfl}
   \frac{(a + \sqrt{2}) \dt}{\dx} < 1. 
  \end{align}
    \end{corollary}

  \begin{proof}
   Consider expression \eqref{eq:cflii} and plug in the definition of both $\psi$ from \eqref{eq:psi} and $\nu_1$ from \eqref{eq:cfli}. We consider case $\eps > \varphi(a)$ first, and obtain
   \begin{align*}
   \left(\frac{\widehat \alpha + \widetilde \alpha}{a} \right) \left( \frac{a}{a+\sqrt{2}} \right)^2 
   & \geq \frac{(a + \sqrt{2})a}{(a+\sqrt{2})^2}\geq \frac{a}{a + \sqrt{2}} . 
   \end{align*}
   Ergo,   $\widehat \nu < \frac{a}{a + \sqrt{2}}$
   is sufficient for \eqref{eq:cflii}, which implies \eqref{eq:regcfl}. Similarly, for $\eps \leq \varphi(a)$, one can easily show that $\widehat \nu < \nu_1$ is fulfilled given that \eqref{eq:regcfl} holds.\qed
  \end{proof}



\begin{figure}[h]
 \begin{center}
   {\includegraphics{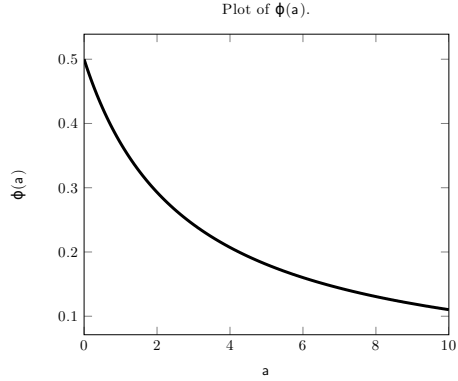}}
   \caption{Plot of function $\varphi$ from \eqref{eq:varphi}.}\label{fig:varphi} \end{center}
\end{figure}

 \begin{remark}
 \begin{itemize}
  \item The function $\varphi(a) := \frac{\sqrt{2}}{a + 2 \sqrt{2}}$, see \eqref{eq:varphi},  has been plotted in Figure~\ref{fig:varphi}. 
  Consider the stability constraint \eqref{eq:cflii}, together with the definition of $\psi(\eps, a)$ in \eqref{eq:psi}. 
  From Figure \ref{fig:varphi}, one can tell that $\eps \leq \varphi(a)$ in \eqref{eq:psi} is met for a sufficiently small $\eps$. This then again implies that for 'small' $\eps$, one has  stability that only depends on the convective $\cfl$ number $\nu_1$, see \eqref{eq:cfli}, as $\psi(\eps) = 1$ for this case. This is an impressive result in the sense that the only stability restriction for small $\eps$ depends on the slow waves. 
  \item In Figure \ref{fig:char_cfl}, we plotted the numerically determined maximum allowable $\widehat \nu_{num}$ values such that the real parts of the eigenvalues of $\Aa$ are negative. For the particular computations, we choose $a = 2, \dx = 10^{-2}$, $\dt = \frac{\widehat \nu_{num} \dx}{a }$, $\widehat \alpha = 2 + \sqrt{2}$, $\widetilde \alpha = 0$ or $\widetilde \alpha= \frac{\sqrt{2}(1-\eps)}{\eps}$ and determine the maximum $\widehat \nu_{num}$, such that $\Real(\eigAk_{k,0}) < 0$ and $\Real(\eigAk_{k,\pm}) < 0$
  One can infer from this figure that the stability restriction one has to impose on the ratio $\frac{\dt}{\dx}$ can be made independent of $\eps$. 
  Note that if there is no known explicit formula for the eigenvalues of $\Aa$, one could still investigate linearized splittings of, e.g., the Euler equations by simply computing all $\Aa$ up to a given value of $k$, getting a first glimpse of possible (in)stabilities in the splitting.
 \end{itemize}
 \end{remark}

\begin{figure}[h]
 \begin{center}
   {\includegraphics[width=0.48\textwidth]{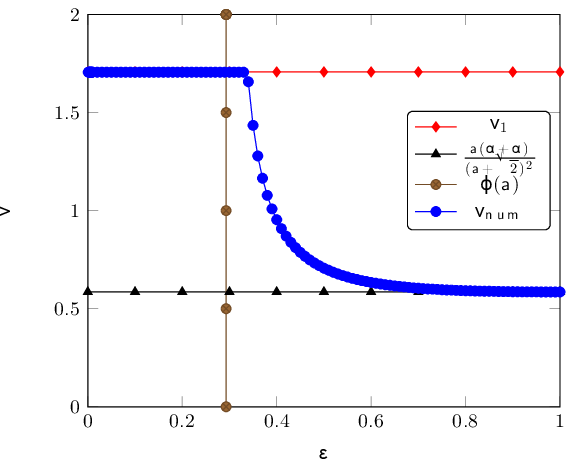}}
   {\includegraphics[width=0.48\textwidth]{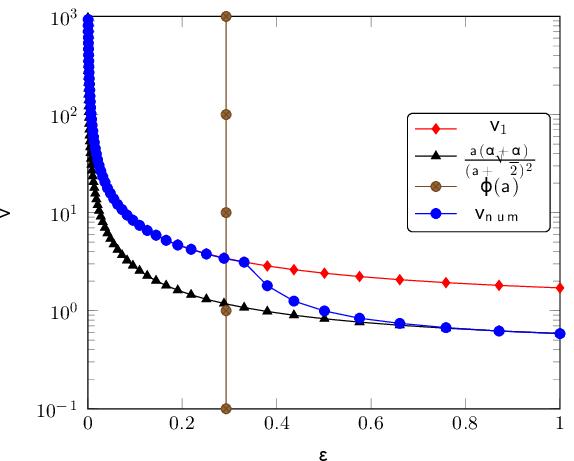}}
     \caption{Determined $\cfl$ number versus a-priori estimates. Left: $\widetilde\alpha =  0$, Right: $\widetilde \alpha = \frac{\sqrt{2}(1-\eps)}{\eps}$.}\label{fig:char_cfl}
 \end{center}
\end{figure}

\section{On The Non-Uniform Stability Of Some Splittings}\label{sec:nonuniformity}
In this section, we consider a splitting that does not induce a scheme that is uniformly stable. This means that there is no bound on 
$\cflhat$,
independently of $\eps$, such that $\Real(\eigAk_{k,i}) < 0$ holds for all $\ifourier \in \Z^{\neq 0}$ and all $\eps > 0$.

In particular, we consider the (non-characteristic) splitting of matrix $A$ from \eqref{eq:3dA} into 
  \begin{align*}
   \widehat A = \left( \begin{matrix} a & \,1\!-\!\eps\, & 0 \\ 1 & a & 1 \\ 0 & 1\!-\!\eps & a\end{matrix} \right), \qquad 
   \widetilde A = \left( \begin{matrix} 0 & \eps & 0 \\ \frac{1-\eps^2}{\eps^2} & \,0\, & \frac{1-\eps^2}{\eps^2} \\ 0 & \eps & 0 \end{matrix} \right). 
  \end{align*}
The eigenvalues of the splitting matrices are 
\begin{alignat*}{2}
 \widehat \lambda_1   &= a, \quad \widehat \lambda_{2,3}   &\ =\ & a \pm \sqrt{2-2\eps} \\
 \widetilde \lambda_1 &= 0, \quad \widetilde \lambda_{2,3} &\ =\ & \pm \frac{\sqrt{2\eps(1-\eps^2)}}{\eps}.
\end{alignat*}

Obviously, this only yields a hyperbolic splitting in the sense of Definition \ref{def:splitting} if we restrict ourselves to $\eps < 1$. (We are only interested in the case $\eps \rightarrow 0$, so this can be done without loss of generality, as it could also be circumvented by a reparametrization of $\eps$.)  For $\eps < 1$, the splitting is admissible. 

The frequency matrix $\Aa$ in this case has eigenvalues which can only be computed via an extremely tedious calculation, or with the aid of Maple. Expanded in terms of the asymptotic sequence $\{\eps^{-2}, \eps^{-1}, \ldots \}$, these eigenvalues are given by
\begin{align}
  \eigAk_{k,0}   &= 2 \pi^2 k^2 \dt a^2 -2 \pi^2 k^2 (\widehat \alpha + \widetilde \alpha) \dx - 2 \pi k i a  + O(\eps), \notag \\ 
  \eigAk_{k,+}   &= \frac{4\pi^2 k^2\dt }{\eps^2} - \frac{8\pi^2k^2 \dt}{\eps} + O(1), \qquad \eigAk_{k,-} = - \frac{4\pi^2 k^2\dt }{\eps^2} + O(1).
\label{eq:muk_pm}
\end{align}
Note that $\Real(\mu_{k,+}) < 0$ for $\eps \rightarrow 0$ can only hold if $\dt = O(\eps)$ (and so the asymptotic expansion is not valid anymore, because $O(1)$ refers to $O(1)$ with respect to $\eps$, not with respect to $\dt$). Consequently, a CFL condition independently of $\eps$ can \emph{not} hold, although we treat the 'fast' parts implicitly. 

\begin{remark}
\begin{itemize} 
 \item This result is in contrast to common belief that coupling two schemes that are individually stable, does indeed yield a stable scheme.
 \item In particular, recall the two forms of the frequency matrix $\Aa$ in \eqref{eq:aa} and \eqref{eq:aa2}. They differ by the commutator
\begin{align*}
  - 2 \pi^2 \ifourier^2 {\dt} \left(
  \widetilde A \widehat A - \widehat A \widetilde A \right),
\end{align*}
whose eigenvalues are
\begin{align*}
  0, \;\; \pm (4 \pi^2 \ifourier^2 {\dt})\frac{1-\eps-\eps^2}{\eps^2}
    \; =  \; \pm \frac{4 \pi^2 \ifourier^2 {\dt}}{\eps^2} + O(\eps^{-1}),
\end{align*}  
and this is precisely the leading order term of $\eigAk_{k,\pm}$ in \eqref{eq:muk_pm}.
This seems to indicate the importance of the commutator, and the difficulty to control its contribution to the frequency matrix.
 \item Our result is a consequence of the fact that for two matrices $A$ and $B$, there is no bound on the eigenvalues of $A\cdot B$ in terms of products of eigenvalues of $A$ and $B$. In our example, the eigenvalues of the commutator are asymptotically larger than the product of those of $\widehat A$ and $\widetilde A$. More precisely, 
\begin{align*}
  \frac{4 \pi^2 \ifourier^2 \dt}{\eps^2} \gg
  |\widehat \lambda_{2,3}| \, |\widetilde \lambda_{2,3}|
  = \frac{\sqrt{2}(a+\sqrt{2})}{\sqrt{\eps}} + O(\sqrt{\eps}).
 \end{align*}
 \item Once more, the characteristic splitting removes the commutator.  
\end{itemize}
\end{remark}

\section{On The Non-Uniform Stability Of Splittings For The Linearized Euler Equations}
In this section, we apply the theory developed earlier to the linearized Euler equations. 
\subsection{Problem statement and analysis}
From now on, we consider $\gamma$ to be the fixed constant $\gamma = 1.4$. Considering a particular simple state 
\begin{align}
 v_0 := (\rho_0, \rho_0 u_0, E_0) := (1, 1, 1),
\end{align}
and setting $A := f'(v_0)$, where $f$ is defined in \eqref{eq:f_euler}, we obtain the linearized system \eqref{eq:underlyingequation} with matrix
\begin{align*}
 A = 
 \left( \begin {array}{ccc} 0&1&0\\ \noalign{}-\frac32+\frac12\,{\it 
\gamma}&3-{\it \gamma}& \,{\frac {{\it \gamma}- 1}{{{\it \eps}}^{2}}}
\\ \noalign{} \,{\it \gamma}\,{{\it \eps}}^{2}- \,{{\it \eps
}}^{2}- \,{\it \gamma}&{\it \gamma}- \frac32\,{\it \gamma}\,{{\it \eps
}}^{2}+ \frac32\,{{\it \eps}}^{2}&{\it \gamma}\end {array} \right) = \frac 1 {5} \left( \begin{matrix} 
     0 & 5 & 0 \\ -4 & 8 & \frac{2}{\eps^2} \\ 2\eps^2-7 & 7-3\eps^2 &7
 \end{matrix}
     \right).
\end{align*}
Its eigenvalues are 
\begin{align}
 \lambda_1 &= 1 \\
 \lambda_{\pm} &= 1 \pm \frac{\sqrt{\gamma (\gamma-1)(1-\frac{\eps^2}{2})}}{\eps} = 1 \pm \frac{\sqrt{0.56(1-\frac{\eps^2}{2})}}{\eps}
\end{align}
and consequently, the associated system of conservation laws fits very nicely into our framework with two fast waves and one slow convective wave. 

We consider a splitting taken from literature \cite{ArNoLuMu12}, which is actually a modification of Klein's splitting \cite{Kl95}. 
On the nonlinear level, it is given by a splitting of the flux function $f(v)$ into the sum of 
  \begin{subequations}
  \begin{align}
   \widehat f(v)   &:= \left( \rho u, \rho u^2 + p,  u (E + \Pi) \right)^T, \\ 
   \widetilde f(v) &:= \left(0, \frac{1-\eps^2}{\eps^2} p, u (p - \Pi) \right)^T. 
  \end{align}
  \end{subequations}
  $\Pi$ is an auxiliary pressure function, and it is defined by 
  \begin{align}
   \Pi(x,t) := \eps^2 p(x,t) + (1-\eps^2) \overline p
  \end{align}
  for a constant value (with respect to space) of $\overline p$. We cannot completely mimic the nonlinear behavior, as this value is often chosen as the infimum of the pressure over the spatial domain. However, we can set it to the constant value of $\overline p = \frac 1 {5}$, which is the pressure for $\eps = 1$, and is the infimum for all $0 \leq \eps \leq 1$.
  Using this (arguably crude) choice, it is straightforward to linearize the splittings, and one obtains the non-stiff matrix
\begin{align}
\label{eq:AhlinEuler}
  \widehat A = \frac 1 {5}\left( \begin{matrix}
			0 & 5 & 0 \\
			-5 + \eps^2 & 10-2\eps^2 & 2 \\
			-6 -\eps^2 + 2\eps^4 & 6 + \eps^2 -3\eps^4 & 5+2\eps^2
		      \end{matrix}
		\right) 
  \end{align}
with eigenvalues 
\begin{align}
 \widehat \lambda_1 &= 1 \\
 \widehat \lambda_{2,3} &= 1 \pm \frac 1 {5} \sqrt{12 - 3\eps^2 - 2\eps^4}
\end{align}
and the corresponding stiff matrix
\begin{align}
 \label{eq:AslinEuler}
 \widetilde A = \frac 1 {5}\left( \begin{matrix}
                       0 & 0 & 0 \\
                       1-\eps^2 & -2 + 2\eps^2 &-\frac{2(\eps^2-1)}{\eps^2} \\
                       -1+3\eps^2-2\eps^4 & 1-4\eps^2 + 3\eps^4 & 2-2\eps^2
                       \end{matrix}
		\right)
\end{align}
with eigenvalues 
\begin{align}
 \widetilde \lambda_1     &= 0 \\
 \widetilde \lambda_{2,3} &= \pm \frac{(\eps^2-1)\sqrt{2-2\eps^2}}{5\eps}.
\end{align}
Using Maple, we can easily evaluate the eigenvalues $\mu$ of the frequency matrix $\Aa$ and approximately writing them as
\begin{align}
 \eigAk_{k,0} 	&= O(1)\\
 \eigAk_{\pm}	&\approx \frac{\left(-1.579136704 \pm 9.474820224 \right) k\dt}{\eps^2} + O(\eps^{-1}).
\end{align}
The same results as in Section \ref{sec:nonuniformity} holds, and $\Real(\eigAk_{k,+})<0$ can only hold for $\dt = O(\eps)$. 

\subsection{Numerical Results}

In this section, we substantiate the results from the previous subsection with suitable numerical experiments. The setup is as before on domain $\Omega = [0, 1]$, and we consider the linearized splitting defined by the matrices $\widehat A$ and $\widetilde A$ in \eqref{eq:AhlinEuler} and \eqref{eq:AslinEuler}, respectively. In addition, we consider a characteristic splitting, where $\widehat \Lambda$ is defined as the diagonal matrix with eigenvalues corresponding to $\eps = 1$. Initial data are given in the \emph{characteristic} variables $w$ as 
\begin{align}
 w(x, 0) = (\cos(4\pi x), 0, 0)^T
\end{align}
where the first component corresponds to the 'slow' eigenvalue. Note that with these initial conditions, it is guaranteed that there is a limit solution for $\eps \rightarrow 0$. 

Numerical results are shown in Figure \ref{fig:comparisonCharArun}. Those results have been computed with the set of parameters as given in Table \ref{tbl:parameters}. Note that the choice of $\dt$ corresponds to a non-stiff cfl number of $\frac{1.53}{10} = 0.153$ for the characteristic splitting, and approximately $\frac{1.7}{10} = 0.17$ for the linearized splitting.

\begin{table}[!ht]
 \begin{tabular}{|c|c | c| c|c|c|}
  \hline
  $\widehat \alpha$ & $\widetilde \alpha$ & $\dx$ & $\dt$ & $T_{end}$ \\ \hline 
  Maximum absolute eigenvalues of $\widehat A$ & 0 & $1/200$ & $10^{-1} \dx$ & 0.1 \\
  \hline
 \end{tabular}
 \caption{Parameters used for the computations in Figure \ref{fig:comparisonCharArun}.}\label{tbl:parameters}

\end{table}

  \begin{figure}
    \begin{center}
    \begin{tikzpicture}[scale=0.5]
      \begin{axis}[scale  only  axis, axis  y  line=left, y  axis  line  style={-}, axis x  line=none, ylabel={$u$, Linearized Splitting}, title={$\eps=10^{-1}$}]
	\addplot[color=blue, mark=none, line width=2pt]  table[x index=0, y index=1] {LinearizedSplittingLinearizedEuler.out};
      \end{axis}
      \begin{axis}[scale  only  axis, axis  y  line=right, y  axis  line  style={-}, xlabel=$x$, ylabel={$u$, Characteristic splitting}]
	\addplot[color=red, mark=none, line width=2pt]  table[x index=0, y index=1] {CharacteristicSplittingLinearizedEuler.out};
      \end{axis}
    \end{tikzpicture}
    \begin{tikzpicture}[scale=0.5]
      \begin{axis}[scale  only  axis, axis  y  line=left, y  axis  line  style={-}, axis x  line=none, ylabel={$u$, Linearized Splitting}, title={$\eps=10^{-3}$}]
	\addplot[color=blue, mark=none, line width=2pt]  table[x index=0, y index=2] {LinearizedSplittingLinearizedEuler.out};
      \end{axis}
      \begin{axis}[scale  only  axis, axis  y  line=right, y  axis  line  style={-}, xlabel=$x$, ylabel={$u$, Characteristic splitting}]
	\addplot[color=red, mark=none, line width=2pt]  table[x index=0, y index=2] {CharacteristicSplittingLinearizedEuler.out};
      \end{axis}
    \end{tikzpicture}
    \begin{tikzpicture}[scale=0.5]
      \begin{axis}[scale  only  axis, axis  y  line=left, y  axis  line  style={-}, axis x  line=none, ylabel={$u$, Linearized Splitting}, title={$\eps=10^{-5}$}]
	\addplot[color=blue, mark=none, line width=2pt]  table[x index=0, y index=3] {LinearizedSplittingLinearizedEuler.out};
      \end{axis}
      \begin{axis}[scale  only  axis, axis  y  line=right, y  axis  line  style={-}, xlabel=$x$, ylabel={$u$, Characteristic splitting}]
	\addplot[color=red, mark=none, line width=2pt]  table[x index=0, y index=3] {CharacteristicSplittingLinearizedEuler.out};
      \end{axis}
    \end{tikzpicture}
    \begin{tikzpicture}[scale=0.5]
      \begin{axis}[scale  only  axis, axis  y  line=left, y  axis  line  style={-}, axis x  line=none, ylabel={$u$, Linearized Splitting}, title={$\eps=10^{-7}$}]
	\addplot[color=blue, mark=none, line width=2pt]  table[x index=0, y index=4] {LinearizedSplittingLinearizedEuler.out};
      \end{axis}
      \begin{axis}[scale  only  axis, axis  y  line=right, y  axis  line  style={-}, xlabel=$x$, ylabel={$u$, Characteristic splitting}]
	\addplot[color=red, mark=none, line width=2pt]  table[x index=0, y index=4] {CharacteristicSplittingLinearizedEuler.out};
      \end{axis}
    \end{tikzpicture}
            
    \caption{Comparison of classical versus characteristic splitting. Blue: Results based on linearized splitting. Red: Results based on characteristic splitting. Note the different scales in the plot.}\label{fig:comparisonCharArun}
    \end{center}
\end{figure}
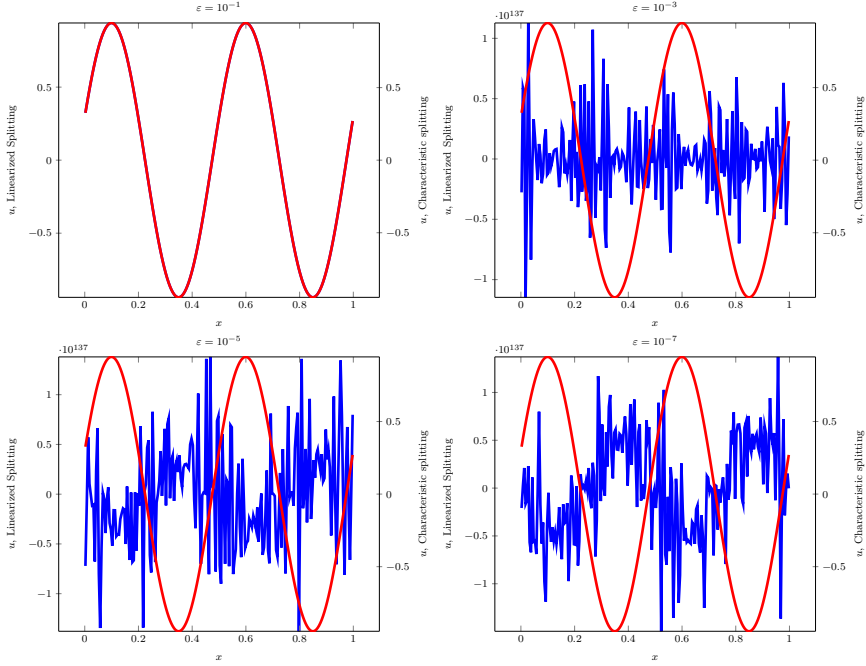

It is clearly visible that the linearized splitting is unstable, at least with the uniform choice of $\dt$ independent of $\eps$, while the characteristic one is not. (Note: The left $y$ axis corresponds to the linearized splitting, and the right one to the characteristic one.)

\section{Conclusions And Outlook}
We developed a technique to investigate the stability and the
largest allowable time steps for low-order IMEX schemes based on a
general class of splittings for linear hyperbolic conservation
laws. 
The eigenvalue analysis reveals the subtle interplay of terms stemming from the discretization of both advection and diffusion, and an additional term stemming from truncation errors in time.

Our analysis, in contrast to common belief, shows that the nonstiff CFL number is usually \emph{not} enough to ensure stability of an IMEX scheme. Indeed, for a splitting introduced earlier, the analysis shows that there is a time step restriction of (at least) order $\eps$ to ensure stability, so the resulting algorithm is not asymptotically stable.

To circumvent this problem, we introduced a new way of obtaining suitable
splittings via characteristic decomposition of the flux Jacobian. Those splittings are stable under a constraint on the nonstiff CFL number that is \emph{independent} of $\eps$. 
We demonstrated that the splitting does influence the stability of
the resulting method, and therefore, one should put effort into
designing suitable flux splittings.

The extension of this analysis to \emph{nonlinear} systems of
conservation laws is not straightforward. However, considering the
linearized equations, the analysis can be easily used as a guiding
principle, similar as the von-Neumann analysis. 
Another challenge is the treatment of multiple dimensions, as the flux-Jacobians usually do not commute. A suitable extension is subject to current research. 

The results presented in this paper concern only the issue of asymptotic stability.
In future work we will extend this study and include the quality of the approximation of low- and higher-order IMEX schemes for compressible flows. 

\bibliographystyle{spmpsci}


\end{document}